\documentclass[10pt]{article}
\usepackage[margin=26mm]{geometry}

\usepackage{graphicx}
\usepackage{amsmath,amsfonts,amssymb}%
\usepackage{amsthm,mathrsfs}%
\usepackage[T1]{fontenc}
\usepackage[latin1]{inputenc}
\usepackage{tikz}
\usepackage{hyperref}
\usepackage[capitalise]{cleveref}
\usepackage[textsize=scriptsize]{todonotes}

\usetikzlibrary{matrix, arrows, calc, patterns, intersections, decorations.markings, shapes.geometric, chains, decorations.pathmorphing}

\newcommand{\R}{\mathbb R}

\newcommand{\ZZ}{\mathbb{ZZ}}
\newcommand{\Z}{\mathbb{Z}}

\newtheorem{theorem}{Theorem}
\numberwithin{theorem}{section}
\newtheorem{lemma}[theorem]{Lemma}
\newtheorem{example}[theorem]{Example}

\newtheorem{corollary}[theorem]{Corollary}
\newtheorem{definition}[theorem]{Definition}

\newtheorem*{theorem*}{Theorem}

\date{January 21, 2020}
\begin{document}

\title{Stability of interval decomposable persistence modules}
\author{H\aa vard Bakke Bjerkevik}

\maketitle

\begin{abstract}
The algebraic stability theorem for $\mathbb{R}$-persistence modules is a fundamental result in topological data analysis. We present a stability theorem for $n$-dimensional rectangle decomposable persistence modules up to a constant $(2n-1)$ that is a generalization of the algebraic stability theorem, and also has connections to the complexity of calculating the interleaving distance. The proof given reduces to a new proof of the algebraic stability theorem with $n=1$. We give an example to show that the bound cannot be improved for $n=2$. We apply the same technique to prove stability results for zigzag modules and Reeb graphs, reducing the previously known bounds to a constant that cannot be improved, settling these questions.
\end{abstract}

\section{Introduction}
\label{intro}
Persistent homology is a tool in topological data analysis used to determine the structure or shape of data sets. For example, given a point cloud $X \subset \mathbb{R}^n$ sampled from a subspace $S$ of $\mathbb{R}^n$, we want to guess at the homology of $S$, which tells us something about how many ``holes'' $S$ has in various dimensions. We can do this by defining $B(\epsilon)$ to be the union of the (open or closed) balls of radius $\epsilon$ centered at each point in $X$. Calculating homology, we get a group or vector space $H_n(B(\epsilon))$ for each $\epsilon \geq 0$, and the inclusions $B(\epsilon) \hookrightarrow B(\epsilon')$ induce morphisms $H_n(B(\epsilon)) \rightarrow H_n(B(\epsilon'))$ for $\epsilon \leq \epsilon'$. Such a collection of vector spaces and morphisms is called a \textit{persistence module}, or $\R$-module, as the vector spaces are parametrized over $\R$. Under certain assumptions, we can decompose an $\R$-module into \textit{interval modules} \cite{crawley}, which gives us a set of intervals uniquely determining the persistence module up to isomorphism. This set of intervals is the \textit{barcode} of the persistence module. The intervals in the barcode are interpreted as corresponding to possible features of the space $S$, where one might interpret long intervals as more likely to describe actual features of $S$ and short intervals as more likely to be the result of noise in the input data. In other words, we have an algorithm with a data set as input and a barcode as output. As data sets always carry a certain amount of noise, we would like this algorithm to be \textit{stable} in the sense that a little change in the input data, or in the persistence modules, should not result in a big change in the barcode.

We measure the difference between persistence modules with the \textit{interleaving distance} $d_I$, and the difference between barcodes with the \textit{bottleneck distance} $d_B$. Proving stability then becomes a question of proving that the bottleneck distance is bounded by the interleaving distance, i.e. $d_B \leq Cd_I$ for some constant $C$. Stability has been proved for persistence modules over $\mathbb{R}$ in \cite{cohen07,proximity,structure,induced} in what is called the \textit{algebraic stability theorem}, which implies the \textit{isometry theorem} $d_I = d_B$.

Persistence modules can also be parametrized over other posets. A pair of filtrations $f,g: S \to \R$ of a topological space $S$ gives rise to an $\R^2$-module which has a vector space $V_p$ for each point in $p \in\R^2$ and linear maps $V_{(a,b)} \to V_{(c,d)}$ whenever $a\leq c$ and $b \leq d$, for instance by letting $V_{(a,b)}$ be $H_n(f^{-1}(-\infty,a) \cap g^{-1}(-\infty,b))$. Again, inclusions induce the linear maps on homology. With $n$ filtrations instead of $2$, we get an $\R^n$-module. Another example is \emph{zigzag modules}, which are popular objects of study in topological data analysis \cite{carlsson2010zigzag,oudot2015zigzag,kim2017stable}. These can arise from a sequence of subspaces $S_i \subset S$, where we also consider the intersections $S_i \cap S_{i+1}$ (or unions). In this case, we have $$\dots \subseteq S_{i-1} \supseteq S_{i-1} \cap S_i \subseteq S_i \supseteq S_i \cap S_{i+1} \subseteq \dots,$$ which again gives rise to linear maps on homology. Defining interleavings and thus the interleaving distance is trickier than for $\R$-modules, but in fact one can do this by associating $\R^2$-modules called \emph{block decomposable modules} to the zigzag modules. One can also associate block decomposable modules to \emph{Reeb graphs}, which are of interest because of their ability to present geometrical information despite being relatively simple objects. See Section \ref{zigReeb}.

All these examples serve as motivation for why one would like to talk about stability for multi-parameter modules (that is, persistence modules parametrized over $\R^n$ for $n\geq 2$). Unfortunately, no isometry theorem is possible even for general $\R^2$-modules, because there is no nice decomposition theorem like in the one-parameter cases, meaning that $d_B$ is not defined. The block decomposable modules, however, decompose nicely, and $d_B \leq \frac{5}{2}d_I$ has been shown for these \cite{zigzag}. This carries over to stability results for zigzag modules and Reeb graphs.

Our main contribution is a new method of proving stability for interval decomposable modules. We demonstrate several applications of this method. The first is \cref{main}:

\begin{theorem*}
Let $M = \bigoplus_{I \in B(M)} \mathbb{I}^I$ and $N = \bigoplus_{J \in B(N)} \mathbb{I}^J$ be rectangle decomposable $\mathbb{R}^n$-modules. If $M$ and $N$ are $\delta$-interleaved, there exists a $(2n-1)\delta$-matching between $B(M)$ and $B(N)$.
\end{theorem*}

This is a generalization of the algebraic stability theorem for $\R$-modules, which is the case $n=1$. For $n\geq 2$, the result is new. There already exist several proofs of the algebraic stability theorem, but our approach is different from the ones taken before, which allows this more general theorem, as well as the results below. Our method is combinatorial, which in our opinion reflects the true nature of the problem once some of the algebraic technicalities are stripped away. Also, our proof is fairly short in the case $n=1$ compared to earlier proofs of the algebraic stability theorem. In \cref{example3}, we construct rectangle decomposable modules $M$ and $N$ over $\R^2$ for which $d_I(M,N)=1$ and $d_B(M,N)=3$, disproving a conjecture made in an earlier version of \cite{zigzag} claiming that $d_B(M,N) = d_I(M,N)$ holds for all $n$-dimensional interval decomposable modules $M$ and $N$ whose barcodes only contain convex intervals. The example also shows that the bound in the theorem cannot be improved for $n=2$. It is an open question if the bound can be improved for $n\geq 3$.

We do not know of any examples of rectangle decomposable modules arising naturally from real-world data sets. But as we discuss in \cref{complexity}, there is a strong link between the stability of these modules and the recent proof that calculating the interleaving distance between multi-parameter modules is NP-hard \cite{bjerkevik2018computing}. In particular, our way of viewing interleavings as pairs of matrices and our observation in \cref{example3} that the interleaving and bottleneck distances differ for rectangle decomposable modules served as inspiration for the approach used in \cite{bjerkevik2018computing}. The question of whether the hardness results can be strengthened is closely related to the question of whether \cref{main} can be improved. Thus, even if rectangle decomposable modules never arise directly from data sets, the type of questions we consider can have an impact on practical applications.

Another reason why we give the proof in detail for rectangle decomposable modules instead of, say, block decomposable modules, is that this case demonstrates very well exactly when our method works and when it fails. The lesson to take home is that the method gives a bound $d_B \leq cd_I$ with a $c$ that increases with the freedom we have in defining the intervals we consider. You need $2n$ coordinates to define an $n$-dimensional (hyper)rectangle, which gives a constant $2n-1$ in the theorem.

Another application of our proof method gives \cref{triangle}:
\begin{theorem*}
Let $M$ and $N$ be $\delta$-interleaved triangle decomposable modules. Then there is a $\delta$-matching between $B(M)$ and $B(N)$.
\end{theorem*}
This is more immediately connected to practical applications. \cref{triangle} implies $d_B\leq d_I$ for block decomposable modules, which is an improvement on the previous best known bound, $d_B\leq \frac{5}{2}d_I$. Since the opposite inequality $d_I\leq d_B$ holds trivially, our bound is the best possible. We discuss how stability results for zigzag modules and Reeb graphs follow in \cref{zigReeb}. The fact that our bound is optimal means that these stability problems are now settled.

We finish off \cref{rectangles} by showing stability for free modules.

We assume that all modules are pointwise finite dimensional (p.f.d.). In a previous version of this paper \cite{myself}, we strengthened the theorems by removing this assumption.

\section{Persistence modules, interleavings, and matchings}
\label{defs}
In this section we introduce some basic notation and definitions that we will use throughout the paper. Let $k$ be a field that stays fixed throughout the text, and let $\textbf{vec}$ be the category of finite dimensional vector spaces over $k$. We identify a poset with its poset category, which has the elements of the poset as objects, a single morphism $p\to q$ if $p\leq q$ and no morphism if $p\nleq q$.

\begin{definition}
Let $P$ be a poset category. A $P$-\textbf{persistence module} is a functor $P \rightarrow \textbf{vec}$.
\end{definition}

If the choice of poset is obvious from the context, we usually write `persistence module' or just `module' instead of `$P$-persistence module'. 

For a persistence module $M$ and $p \leq q \in P$, $M(p)$ is denoted by $M_p$ and $M(p \rightarrow q)$ by $\phi_M(p,q)$. We refer to the morphisms $\phi_M(p,q)$ as the \textit{internal morphisms} of $M$.~$M$ being a functor implies that $\phi_M(p,p) = id_{M_p}$, and that $\phi_M(q,r) \circ \phi_M(p,q) = \phi_M(p,r)$. Because the persistence modules are defined as functors, they automatically assemble into a category where the morphisms are natural transformations. This category is denoted by $P\textbf{-mod}$. Let $f: M \rightarrow N$ be a morphism between persistence modules. Such an $f$ consists of a morphism associated to each $p \in P$, and these morphisms are denoted by $f_p$. Because $f$ is a natural transformation, we have $\phi_N(p,q) \circ f_p = f_q \circ \phi_M(p,q)$ for all $p \leq q$.

\begin{definition}
An \textbf{interval} is a subset $\varnothing \neq I \subseteq P$ that satisfies the following:
\begin{itemize}
\item If $p,q \in I$ and $p \leq r \leq q$, then $r \in I$.
\item If $p,q \in I$, then there exist $p_1, p_2, \dots, p_{2m} \in I$ for some $m \in \mathbb{N}$ such that $p \leq p_1 \geq p_2 \leq \dots \geq p_{2m} \leq q$.
\end{itemize}
\end{definition}
We refer to the last point as the \textit{connectivity axiom} for intervals.

\begin{definition}
An \textbf{interval persistence module} or \textbf{interval module} is a persistence module $M$ that satisfies the following: for some interval $I$,~$M_p = k$ for $p \in I$ and $M_p = 0$ otherwise, and $\phi_M(p,q) = \textrm{Id}_k$ for points $p \leq q$ in $I$. We use the notation $\mathbb{I}^J$ for the interval module with $J$ as its underlying interval.
\end{definition}

The definitions up to this point have been valid for all posets $P$, but we need some additional structure on $P$ to get a notion of distance between persistence modules, which is essential to prove stability results. Since we will mostly be working with $\mathbb{R}^n$-persistence modules, we restrict ourselves to this case from now on. We define the poset structure on $\mathbb{R}^n$ by letting $(a_1, a_2, \dots, a_n) \leq (b_1, b_2, \dots, b_n)$ if and only if $a_i \leq b_i$ for $1 \leq i \leq n$. For $\epsilon \in \mathbb{R}$, we often abuse notation and write $\epsilon$ when we mean $(\epsilon,\epsilon,\dots,\epsilon) \in \mathbb{R}^n$. We call an interval $I \subset \mathbb{R}^n$ \textit{bounded} if it is bounded as a subset of $\mathbb{R}^n$ in the usual sense. That is, it is contained in a ball with finite radius.

\begin{definition}
For $\epsilon \in [0,\infty)$, we define the \textbf{shift functor} $(\cdot)(\epsilon): \mathbb{R}^n \textbf{-mod} \rightarrow \mathbb{R}^n \textbf{-mod}$ by letting $M(\epsilon)$ be the persistence module with $M(\epsilon)_p = M_{p+\epsilon}$ and $\phi_{M(\epsilon)}(p,q) = \phi_M(p+\epsilon,q+\epsilon)$. For morphisms $f: M \rightarrow N$, we define $f(\epsilon): M(\epsilon) \rightarrow N(\epsilon)$ by $f(\epsilon)_p = f_{p+\epsilon}$.
\end{definition}

We also define shift on intervals $I$ by letting $I(\epsilon)$ be the interval for which $\mathbb{I}^{I(\epsilon)} = \mathbb{I}^I(\epsilon)$.

Define the morphism $\phi_{M,\epsilon}: M \rightarrow M(\epsilon)$ by $(\phi_{M,\epsilon})_p = \phi_M(p,p+\epsilon)$.

\begin{definition}
An \textbf{$\epsilon$-interleaving} between $\mathbb{R}^n$-modules $M$ and $N$ is a pair of morphisms $f: M \rightarrow N(\epsilon)$,~$g: N \rightarrow M(\epsilon)$ such that $g(\epsilon) \circ f = \phi_{M,2\epsilon}$ and $f(\epsilon) \circ g = \phi_{N,2\epsilon}$.
\end{definition}

If there exists an $\epsilon$-interleaving between $M$ and $N$, then $M$ and $N$ are said to be $\epsilon$-interleaved. An interleaving can be viewed as an `approximate isomorphism', and a $0$-interleaving is in fact an isomorphism. We call a module $M$ \textit{$\epsilon$-significant} if $\phi_M(p,p+\epsilon) \neq 0$ for some $p$, and \textit{$\epsilon$-trivial} otherwise.~$M$ is $2\epsilon$-trivial if and only if it is $\epsilon$-interleaved with the zero module. We call an interval $I$ $\epsilon$-significant if $\mathbb{I}^I$ is $\epsilon$-significant, and $\epsilon$-trivial otherwise.

\begin{definition}
We define the \textbf{interleaving distance} $d_I$ on persistence modules $M$ and $N$ by

\begin{equation}
d_I(M,N) = \inf\{\epsilon \mid M \textrm{ and } N \textrm{ are } \epsilon\textrm{-interleaved}\}.
\end{equation}
\end{definition}
The interleaving distance intuitively measures how close the modules are to being isomorphic. The interleaving distance between two modules might be infinite, and the interleaving distance between two different, even non-isomorphic modules, might be zero. Apart from this, $d_I$ satisfies the axioms for a metric, so $d_I$ is an extended pseudometric.

\begin{definition}
Suppose $M \cong \bigoplus_{I \in B} \mathbb{I}^I$ for a multiset\footnote{We will not be rigorous in our treatment of multisets. A multiset may contain multiple copies of one element, but we will assume that we have some way of separating the copies, so that we can treat the multiset as a set. If e.g.~$I$ and $J$ are intervals in a multiset and we say that $I \neq J$, we mean that they are ``different'' elements of the multiset, not that they are different intervals.} $B$ of intervals. Then we call $B$ the \textbf{barcode} of $M$, and write $B = B(M)$. We say that $M$ is \textbf{interval decomposable}.
\end{definition}

Since the endomorphism ring of any interval module is isomorphic to $k$, it follows from Theorem $1$ in \cite{azumaya} that if a persistence module $M$ is interval decomposable, the decomposition is unique up to isomorphism. Thus the barcode is well-defined, even if we let $M$ be a $P$-module for an arbitrary poset $P$. If $M$ is a p.f.d.~$\mathbb{R}$-module, it is interval decomposable \cite{crawley}, but this is not true for $\mathbb{R}$-modules or p.f.d.~$\mathbb{R}^n$-modules in general. \cite{webb} gives an example showing the former, and the following is an example of a $P$-module for a poset $P$ with four points that is not interval decomposable.
\begin{equation}
\begin{tikzpicture}[baseline=(current  bounding  box.center)]
\matrix(m)[matrix of math nodes, row sep=4em, column sep=4em, text height=1.5ex, text depth=0.25ex]
{k&k^2&k\\
&k&\\};
\path[->,font=\scriptsize,>=angle 90]
(m-1-1) edge node[auto] {$\begin{pmatrix}1 & 0\end{pmatrix}$} (m-1-2)
(m-1-2) edge node[auto] {$\begin{pmatrix}1 \\ 1\end{pmatrix}$} (m-1-3)
(m-2-2) edge node[auto] {$\begin{pmatrix}0 & 1\end{pmatrix}$} (m-1-2);
\end{tikzpicture}
\end{equation}
A corresponding $\mathbb{R}^2$-module that is not interval decomposable and is at most two-dimensional at each point can be constructed.

For multisets $A,B$, we define a \textit{partial bijection} as a bijection $\sigma: A' \rightarrow B'$ for some subsets $A' \subset A$ and $B' \subset B$, and we write $\sigma: A \nrightarrow B$. We write $\textrm{coim } \sigma = A'$ and $\textrm{im } \sigma = B'$.

\begin{definition}
Let $A$ and $B$ be multisets of intervals. An \textbf{$\epsilon$-matching} between $A$ and $B$ is a partial bijection $\sigma: A \nrightarrow B$ such that
\begin{itemize}
\item all $I \in A \setminus \textrm{coim } \sigma$ are $2\epsilon$-trivial
\item all $I \in B \setminus \textrm{im } \sigma$ are $2\epsilon$-trivial
\item for all $I \in \textrm{coim } \sigma$, $\mathbb{I}^I$ and $\mathbb{I}^{\sigma(I)}$ are $\epsilon$-interleaved.
\end{itemize}
\end{definition}
If there is an $\epsilon$-matching between $B(M)$ and $B(N)$ for persistence modules $M$ and $N$, we say that $M$ and $N$ are \textit{$\epsilon$-matched}.

We have adopted this definition of $\epsilon$-matching from \cite{zigzag}, which differs from e.g. the one in \cite{structure}, which allows two intervals $I$ and $J$ to be matched if $d_I(\mathbb{I}^I,\mathbb{I}^J) \leq \epsilon$ (or rather, this is equivalent to their definition). Conveniently, with the definition we have chosen, an $\epsilon$-interleaving is easily constructed given an $\epsilon$-matching. We feel that this is the more natural definition for this paper, as several of our results are phrased as statements about matchings and interleavings, and the interleaving distance might not come into the picture at all. The other definition is perhaps more natural in the context of `persistence diagrams', where intervals are identified with points in a diagram, and the interleaving distance between the corresponding modules is simply the distance between the points. This is irrelevant to us, however, as we never consider persistence diagrams.

We can also define $\epsilon$-matchings in the context of graph theory. A \textit{matching in a graph} is a set of edges in the graph without common vertices, and a matching is said to \textit{cover} a set $S$ of vertices if all elements in $S$ are adjacent to an edge in the matching. Let $G_\epsilon$ be the bipartite graph on $A \sqcup B$ with an edge between $I \in A$ and $J \in B$ if $\mathbb{I}^I$ and $\mathbb{I}^J$ are $\epsilon$-interleaved. Then an $\epsilon$-matching between $A$ and $B$ is a matching in $G_\epsilon$ such that the set of $2\epsilon$-significant intervals in $A \sqcup B$ is covered.

\begin{definition}
The \textbf{bottleneck distance} $d_B$ is defined by
\begin{equation}
d_B(M,N) = \inf\{\epsilon \mid M \textrm{ and } N \textrm{ are } \epsilon \textrm{-matched}\}
\end{equation}
for any interval decomposable $M$ and $N$.
\end{definition}
We might abuse notation and talk about $d_B(C,D)$, where $C$ and $D$ are barcodes.

\section{Zigzag modules and Reeb graphs}
\label{zigReeb}

In this section we will give some intuition for how block decomposable modules relate to Reeb graphs and zigzag modules. We refer to \cite{zigzag} for a more detailed and rigorous treatment.

When explaining the connection to Reeb graphs and zigzag modules, it is more convenient to flip one of the axes in $\R^2$, so that we work with $\R^{\text{op}} \times \R$ instead. This way, $(a,b) \leq (c,d)$ iff $c \leq a$ and $b \leq d$, or, equivalently, if $(a,b) \subseteq (c,d)$ as intervals, assuming $a<b$. Let $\mathbb U = \{(a,b) \in \R^{\text{op}} \times \R \mid a \leq b \}$.

\begin{definition}
An interval decomposable $\R^{\text{op}} \times \R$-module is called \textbf{block decomposable} if its barcode only contains intervals of the following types:
\begin{itemize}
\item $[a,b]_{\text{BL}} = \{(c,d) \in \mathbb U \mid c \leq b, d \geq a\}$
\item $[a,b)_{\text{BL}} = \{(c,d) \in \mathbb U \mid a \leq d < b\}$
\item $(a,b]_{\text{BL}} = \{(c,d) \in \mathbb U \mid a < c \leq b\}$
\item $(a,b)_{\text{BL}} = \{(c,d) \in \mathbb U \mid c > a, d < b\}$
\end{itemize}
\end{definition}
We call these intervals \emph{blocks}. Each interval intersects the diagonal in an $\R$-interval that is open, closed or half-open one way or the other depending on the type of the block.

\begin{figure}
\centering
\begin{tikzpicture}[scale=.6]
\begin{scope}[xshift=-9cm]
\fill [black,opacity=0.2] (-1,-1)--(-2,-1)--(-2,2)--(1,2)--(1,1);
\draw[<->,thick] (-2,-2) to (2,2);
\node at (-.5,-1.3){$(a,a)$};
\node at (1.5,.7){$(b,b)$};
\end{scope}
\begin{scope}[xshift=-3cm]
\fill [black,opacity=0.2] (-1,-1)--(-2,-1)--(-2,1)--(1,1);
\draw[<->,thick] (-2,-2) to (2,2);
\node at (-.5,-1.3){$(a,a)$};
\node at (1.5,.7){$(b,b)$};
\end{scope}
\begin{scope}[xshift=3cm]
\fill [black,opacity=0.2] (-1,-1)--(-1,2)--(1,2)--(1,1);
\draw[<->,thick] (-2,-2) to (2,2);
\node at (-.5,-1.3){$(a,a)$};
\node at (1.5,.7){$(b,b)$};
\end{scope}
\begin{scope}[xshift=9cm]
\fill [black,opacity=0.2] (-1,-1)--(-1,1)--(1,1);
\draw[<->,thick] (-2,-2) to (2,2);
\node at (-.5,-1.3){$(a,a)$};
\node at (1.5,.7){$(b,b)$};
\end{scope}
\end{tikzpicture}
\caption{The intervals $[a,b]_{\text{BL}}, [a,b)_{\text{BL}}, (a,b]_{\text{BL}}$ and $(a,b)_{\text{BL}}$.\label{blocktypes}}
\end{figure}

\subsection{Reeb graphs}

There have been proposed several distances on Reeb graphs; see \cite{bauer2014measuring} for a summary, as well as references to various applications. The interleaving distance we consider was introduced in \cite{de2016categorified}.

A Reeb graph is a topological graph $G$ together with a continuous function $\gamma: G \to \R$ such that the level sets of $\gamma$ are discrete. Let $S(\gamma) = \R^{\text{op}} \times \R \to \textbf{Set}$ be the functor sending $(a,b)$ to the set of connected components of $\gamma^{-1}(a,b)$ and $S((a,b) \to (c,d))$ be induced by the inclusion $\gamma^{-1}(a,b) \subseteq \gamma^{-1}(c,d)$ for $c \leq a \leq b \leq d$. In Figure \ref{Reeb}, $\gamma$ is projection to a horizontal axis. Above the graph, the functor $S(\gamma)$ is shown, the shade of grey at $(a,b)$ determined by the size of $S(\gamma)_{(a,b)}$.

Given two Reeb graphs $(G_1,\gamma_1)$ and $(G_2,\gamma_2)$, we get two functors $S(\gamma_1)$ and $S(\gamma_2)$, and we can talk about interleavings and interleaving distance by adjusting the definitions in the previous section. It turns out that this interleaving distance is at least as big as the one we get by replacing $S(\gamma_1)$ and $S(\gamma_2)$ by corresponding block decomposable modules $M_{\gamma_1}$ and $M_{\gamma_2}$. In Figure \ref{Reeb}, the blocks comprising this block decomposable module are exactly what you would guess by looking at the figure. In \cite{zigzag}, $d_B(M,N) \leq \frac{5}{2} d_I(M,N)$ is proved for such modules; with Theorem \ref{thmBlock}, we have $d_B(M,N) = d_I(M,N)$.

There is also a barcode $L_0(\gamma)$ of $\R$-intervals (the level set persistence diagram \cite{carlsson2009zigzag}) associated to a Reeb graph $(G,\gamma)$, which we can think of as arising from the intersection of $S(\gamma)$ with the diagonal $x=y$. This barcode is the same as $B(M_\gamma)$, except that $(a,b]_{BL}$ is replaced by $(a,b]$, and so on. It is not too hard to see that $d_B(L_0(\gamma_1),L_0(\gamma_2)) \leq 2d_B(M_{\gamma_1},M_{\gamma_2})$.\footnote{The reason for the constant $2$ is that $(a,b)_{BL}$ is $(b-a)/4$-trivial, while $(a,b)$ is not $\epsilon$-trivial for $\epsilon < (b-a)/2$.} Altogether, this gives
\begin{align*}
d_B(L_0(\gamma_1),L_0(\gamma_2)) &\leq 2d_B(M_{\gamma_1},M_{\gamma_2})\\
&= 2d_I(M_{\gamma_1},M_{\gamma_2})\\
&\leq 2d_I(S(\gamma_1),S(\gamma_2)).
\end{align*}
In other words:
\begin{theorem}
\label{thmReeb}
For Reeb graphs $\gamma_1$, $\gamma_2$, the inequality $d_B(L_0(\gamma_1),L_0(\gamma_2)) \leq 2d_I(S(\gamma_1),S(\gamma_2))$ holds.
\end{theorem}
Thus an easily computed bottleneck distance gives a lower bound for the interleaving distance between Reeb graphs. This improves the result in \cite{zigzag}, which was again an improvement on \cite{bauer2014strong}, by lowering the constant in the inequality from $5$ to $2$, and this cannot be improved.
\begin{figure}
\centering
\begin{tikzpicture}[scale=.6]
\fill [black,opacity=0.2] (-2,-2)--(-4,-2)--(-4,11)--(9,11)--(9,9);
\fill [black,opacity=0.2] (-1,-1)--(-4,-1)--(-4,0)--(0,0);
\fill [black,opacity=0.2] (2,2)--(2,5)--(5,5);
\fill [black,opacity=0.2] (6,6)--(6,11)--(8,11)--(8,8);
\draw[<->,thick] (-4,-4) to (11,11);
\node at (12,11){$x=y$};
\node at (3,4){$\bullet$};
\draw[dotted] (3,4) to (4,4);
\draw[dotted] (3,4) to (3,3);
\fill [red,opacity=0.4] (3,3)--(4,4)--(4,-5)--(3,-5);
\begin{scope}[yshift=-4cm]
\draw[thick] (-2,1) to (0,0);
\draw[thick] (-1,-1) to (0,0);
\draw[thick] (0,0) to (2,0);
\draw[thick] (2,0) to [out=70,in=110] (5,0);
\draw[thick] (2,0) to [out=290,in=250] (5,0);
\draw[thick] (5,0) to (6,0);
\draw[thick] (6,0) to (9,1);
\draw[thick] (6,0) to (8,-.5);
\draw[dotted] (0,0) to (0,4);
\draw[dotted] (-1,-1) to (-1,3);
\draw[dotted] (-2,1) to (-2,2);
\draw[dotted] (2,0) to (2,6);
\draw[dotted] (5,0) to (5,9);
\draw[dotted] (6,0) to (6,10);
\draw[dotted] (8,-.5) to (8,12);
\draw[dotted] (9,1) to (9,13);
\end{scope}
\end{tikzpicture}
\caption{A Reeb graph $(G,\gamma)$ with $S(\gamma)$ above. Evaulating $S(\gamma)$ at the point shown, we get the intersection of $G$ with the red strip, which has two connected components. \label{Reeb}}
\end{figure}
\subsection{Zigzag modules}
A zigzag module is a module over $\ZZ = \{(a,b) \in \Z^2 \mid a=b \vee a=b+1\}$ taken as a sub-poset of $\R^{\text{op}} \times \R$. Let $\ZZ|_{(a,b)}$ be the sub-poset of $\ZZ$ containing the elements $\{(c,d) \in \ZZ \mid a \leq c, d \leq b \}$. A zigzag module $M$ gives rise to a block decomposable module $M_{BL}$ defined by letting $M_{BL}(a,b)$ be the colimit of the restriction of $M$ to $\ZZ|_{(a,b)}$. $M_{BL}((a,b)\to (c,d))$ is defined to be the induced morphism we get by the universal property of colimits for $(a,b) \leq (c,d)$. (This definition is given in \cite{zigzag}, but something very similar is described in the discussions of pyramids in \cite{carlsson2009zigzag} and \cite{bendich2013homology}.) This way, we can define interleaving and bottleneck distance between zigzag modules by letting $d_I(M,N) = d_I(M_{BL},N_{BL})$ and $d_B(M,N) = d_B(M_{BL},N_{BL})$. Thus Theorem \ref{thmBlock} holds if we replace `block decomposable modules' by `zigzag modules':
\begin{theorem}
\label{thmZigzag}
Let $M$ and $N$ be zigzag modules. If $M$ and $N$ are $\delta$-interleaved, there exists a $\delta$-matching between $B(M)$ and $B(N)$.
\end{theorem}
This implies an isometry theorem for zigzag modules: $d_I(M,N) = d_B(M,N)$.

\section{Higher-dimensional stability}
\label{rectangles}

The algebraic stability theorem for $\mathbb{R}$-modules states that an $\epsilon$-interleaving between $\mathbb{R}$-modules $M$ and $N$ induces an $\epsilon$-matching between $B(M)$ and $B(N)$, implying $d_I(M,N) = d_B(M,N)$, the isometry theorem. The main purpose of this paper is to find out when similar results for $\mathbb{R}^n$-modules hold. Our first result, Theorem \ref{main}, is a generalization of the algebraic stability theorem for $\mathbb{R}$-modules. Variations of the algebraic stability theorem have been proved several times already \cite{cohen07,proximity,structure,induced}, but this is a new proof with ideas that are applicable to more than just $\mathbb{R}$-modules.

\subsection{Rectangle decomposable modules}

For any interval $I \subset \mathbb{R}^n$, we let its projection on the $i$'th coordinate be denoted by $I_i$.

\begin{definition}
A \textbf{rectangle} is an interval of the form $R = R_1 \times R_2 \times \dots \times R_n$.
\end{definition}

Two rectangles $R$ and $S$ are of the \textit{same type} if $R_i \setminus S_i$ and $S_i \setminus R_i$ are bounded for every $i$. For $n=1$, we have four types of rectangles:
\begin{itemize}
\item intervals of finite length
\item intervals of the form $(a,\infty)$ or $[a,\infty)$
\item intervals of the form $(-\infty,a)$ or $(-\infty,a]$
\item $(-\infty,\infty)$,
\end{itemize}
for some $a \in \mathbb{R}$. We see that for $n \geq 1$, rectangles $R$ and $S$ are of the same type if $R_i$ and $S_i$ are of the same type for all $1 \leq i \leq n$. Examples of $2$-dimensional rectangles are given in Figure \ref{fig:rectangletypes}.

\begin{figure}
\centering
\includegraphics[scale=0.4]{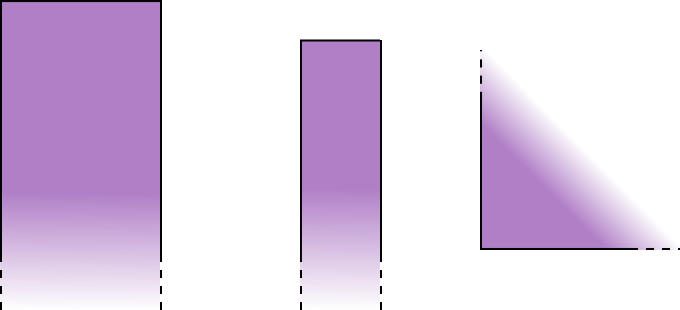}
\caption{Three rectangles, where the left and middle rectangles are of the same type (unbounded downwards), while the last is of a different type (unbounded upwards and to the right). Assuming that it contains its boundary, the rightmost rectangle is also an example of a free interval, which we will define in a later section.\label{fig:rectangletypes}}
\end{figure}

In \cite{structure}, \textit{decorated numbers} were introduced. These are endpoints of intervals `decorated' with a plus or minus sign depending on whether the endpoints are included in the interval or not. Let $\overline{\mathbb{R}} = \mathbb{R} \cup \{ -\infty, \infty\}$. A decorated number is of the form $a^+$ or $a^-$, where $a \in \overline{\mathbb{R}}$.\footnote{The decorated numbers $-\infty^-$ and $\infty^+$ are never used, as no interval contains points at infinity, but it does not matter whether we include these two points in the definition.} The notation is as follows for $a,b \in \overline{\mathbb{R}}$:
\begin{itemize}
\item $I = (a^+, b^+)$ if $I = (a, b]$
\item $I = (a^+, b^-)$ if $I = (a, b)$
\item $I = (a^-, b^+)$ if $I = [a, b]$
\item $I = (a^-, b^-)$ if $I = [a, b)$.
\end{itemize}
We define decorated points in $n$ dimensions for $n \geq 1$ as tuples $a = (a_1, a_2, \dots, a_n)$, where all the $a_i$'s are decorated numbers. For an $n$-dimensional rectangle $R$ and decorated points $(a_1, a_2, \dots, a_n)$ and $(b_1, b_2, \dots, b_n)$, we write $R = ((a_1, a_2, \dots, a_n),$ $(b_1, b_2, \dots, b_n))$ if $R_i = (a_i, b_i)$ for all $i$. We define $\textrm{min}_R$ and $\textrm{max}_R$ as the decorated points for which $R = (\textrm{min}_R, \textrm{max}_R)$. We write $a^*$ for decorated numbers with unknown `decoration', so $a^*$ is either $a^+$ or $a^-$.

There is a total order on the decorated numbers given by $a^* < b^*$ for $a < b$, and $a^- < a^+$ for all $a,b \in \overline{\mathbb{R}}$. This induces a poset structure on decorated $n$-dimensional points given by $(a_1, a_2, \dots, a_n) \leq (b_1, b_2, \dots, b_n)$ if $a_i \leq b_i$ for all $i$. We can also add decorated numbers and real numbers by letting $a^+ + x = (a + x)^+$ and $a^- + x = (a + x)^-$ for $a \in \overline{\mathbb{R}}$, $x \in \mathbb{R}$. We add $n$-dimensional decorated points and $n$-tuples of real numbers coordinatewise.

If $M$ is an interval decomposable $\mathbb{R}^n$-module and all $I \in B(M)$ are rectangles,~$M$ is \textit{rectangle decomposable}.

Our goal is to prove the following theorem:

\begin{theorem}
\label{main}
Let $M = \bigoplus_{I \in B(M)} \mathbb{I}^I$ and $N = \bigoplus_{J \in B(N)} \mathbb{I}^J$ be rectangle decomposable $\mathbb{R}^n$-modules. If $M$ and $N$ are $\delta$-interleaved, there exists a $(2n-1)\delta$-matching between $B(M)$ and $B(N)$.
\end{theorem}
The inequality $d_B(M,N) \leq (2n-1)d_I(M,N)$ for rectangle decomposable modules $M$ and $N$ immediately follows.

Fix $0 \leq \delta \in \mathbb{R}$. Assume that $M$ and $N$ are $\delta$-interleaved, with interleaving morphisms $f: M \rightarrow N(\delta)$ and $g: N \rightarrow M(\delta)$. Recall that this means that $g(\delta) \circ f = \phi_{M, 2 \delta}$ and $f(\delta) \circ g = \phi_{N, 2 \delta}$. For any $I \in B(M)$, we have a canonical injection $\mathbb{I}^I \xrightarrow{\iota_I} M$ and projection $M \xrightarrow{\pi_I} \mathbb{I}^I$, and likewise, we have canonical morphisms $\mathbb{I}^J \xrightarrow{\iota_J} N$ and $N \xrightarrow{\pi_J} \mathbb{I}^J$ for $J \in B(N)$. We define
\begin{align}
\begin{split}
f_{I,J} &= \pi_J(\delta) \circ f \circ \iota_I: \mathbb{I}^I \rightarrow \mathbb{I}^J(\delta) \\
g_{J,I} &= \pi_I(\delta) \circ g \circ \iota_J: \mathbb{I}^J \rightarrow \mathbb{I}^I(\delta).
\end{split}
\end{align}

We prove the theorem by a mix of combinatorial and geometric arguments. First we show that it is enough to prove the theorem under the assumption that all the rectangles in $B(M)$ and $B(N)$ are of the same type. Then we define a real-valued function $\alpha$ on the set of rectangles which in a sense measures, in the case $n=2$, how far `up and to the right' a rectangle is. There is a preorder $\leq_\alpha$ associated to $\alpha$. The idea behind $\leq_\alpha$ is that if there is a nonzero morphism $\chi: \mathbb{I}^I \rightarrow \mathbb{I}^J(\epsilon)$ and $I \leq_\alpha J$, then $I$ and $J$ have to be close to each other. Finding pairs of intervals in $B(M)$ and $B(N)$ that are close is exactly what we need to construct a $(2n-1)\delta$-matching. Lemmas \ref{close} and \ref{nonzero} say that such morphisms behave nicely in a precise sense that we will exploit when we prove Lemma \ref{lemmaHall}. If we remove the conditions mentioning $\leq_\alpha$, Lemmas \ref{close} and \ref{nonzero} are not even close to being true, so one of the main points in the proof of Lemma \ref{lemmaHall} is that we must exclude the cases that are not covered by Lemmas \ref{close} and \ref{nonzero}. We do this by proving that a certain matrix is upper triangular, where the `bad cases' correspond to the elements above the diagonal and the `good cases' correspond to elements on and below the diagonal.

Lemma \ref{lemmaHall} is what ties together the geometric and combinatorial parts of the proof of Theorem \ref{main}. While we prove Lemma \ref{lemmaHall} by geometric arguments, by Hall's marriage theorem the lemma is equivalent to a statement about matchings between $B(M)$ and $B(N)$. We have to do some combinatorics to get exactly the statement we need, namely that there is a $(2n-1)\delta$-matching between $B(M)$ and $B(N)$, and we do this after stating Lemma \ref{lemmaHall}.

We begin by describing morphisms between rectangle modules.

\begin{lemma}
\label{k-endo}
Let $\chi: \mathbb{I}^I \rightarrow \mathbb{I}^J$ be a morphism between interval modules. Suppose $A = I \cap J$ is an interval. Then, for all $a,b \in A$,~$\chi_a = \chi_b$ as $k$-endomorphisms.
\end{lemma}

\begin{proof}
Suppose $a \leq b$ and $a,b \in A$. Then $\chi_b \circ \phi_{\mathbb{I}^I}(a,b) = \phi_{\mathbb{I}^J}(a,b) \circ \chi_a$. Since the $\phi$-morphisms are identities, we get $\chi_a = \chi_b$ as $k$-endomorphisms. By the connectivity axiom for intervals, the equality extends to all elements in $A$.
\end{proof}

Since the intersection of two rectangles is either empty or a rectangle, we can describe a morphism between two rectangle modules uniquely as a $k$-endomorphism if their underlying rectangles intersect. A $k$-endomorphism, in turn, is simply multiplication by a constant. Note that we could have relaxed the assumptions in the proof above and assumed that $a$ is in $I$ instead of in $A$, and still have gotten $\chi_a = \chi_b$. In particular, this means that if $0 \neq \chi: \mathbb{I}^I \rightarrow \mathbb{I}^J$, and $I$ and $J$ are rectangles, then $\textrm{min}_{J_i} \leq \textrm{min}_{I_i}$ for all $i$, which gives $\textrm{min}_J \leq \textrm{min}_I$. Similarly,~$\textrm{max}_J \leq \textrm{max}_I$, and one can also see that $\textrm{min}_I < \textrm{max}_J$ must hold, or else $I \cap J = \varnothing$. We summarize these observations as a corollary of Lemma \ref{k-endo}:

\begin{corollary}
\label{minmaxineq}
Let $R$ and $S$ be rectangles, and let $\chi: \mathbb{I}^R \rightarrow \mathbb{I}^S$ be a nonzero morphism. Then $\textrm{min}_S \leq \textrm{min}_R$ and $\textrm{max}_S \leq \textrm{max}_R$.
\end{corollary}

This will come in handy when we prove Lemmas \ref{typesplit}, \ref{close}, and \ref{nonzero}.

We define a function $w: (B(M) \times B(N)) \sqcup (B(N) \times B(M)) \rightarrow k$ by letting $w(I,J) = x$ if $f_{I,J}$ is given by multiplication by $x$, and $w(I,J) = 0$ if $f_{I,J}$ is the zero morphism.~$w(J,I)$ is given by $g_{J,I}$ in the same way.

With the definition of $w$, it is starting to become clear how combinatorics comes into the picture. We can now construct a bipartite weighted directed graph on $B(M) \sqcup B(N)$ by letting $w(I,J)$ be the weight of the edge from $I$ to $J$. The reader is invited to keep this picture in mind, as a lot of what we do in the rest of the proof can be interpreted as statements about the structure of this graph.

The following lemma allows us to break up the problem and focus on the components of $M$ and $N$ with the same types separately.

\begin{lemma}
\label{typesplit}
Let $R$ and $T$ be rectangles of the same type, and $S$ be a rectangle of a different type. Then $\psi\chi = 0$ for any pair $\chi: \mathbb{I}^R \rightarrow \mathbb{I}^S$,~$\psi: \mathbb{I}^S \rightarrow \mathbb{I}^T$ of morphisms.
\end{lemma}

\begin{proof}
Suppose $\psi, \chi \neq 0$. By Corollary \ref{minmaxineq}, $\textrm{min}_R \geq \textrm{min}_S \geq \textrm{min}_T$ and $\textrm{max}_R \geq \textrm{max}_S \geq \textrm{max}_T$. We get $\textrm{min}_{R_i} \geq \textrm{min}_{S_i} \geq \textrm{min}_{T_i}$ and $\textrm{max}_{R_i} \geq \textrm{max}_{S_i} \geq \textrm{max}_{T_i}$ for all $i$, and it follows that if $R$ and $T$ are of the same type, then $S$ is of the same type as $R$ and $T$.
\end{proof}

Let $f': M \rightarrow N(\delta)$ be defined by $f'_{I,J} = f_{I,J}$ for $I \in B(M)$ and $J \in B(N)$ if $I$ and $J$ are of the same type, and $f'_{I,J} = 0$ if they are not, and let $g': N \rightarrow M(\delta)$ be defined analogously. Here $f'$ and $g'$ are assembled from $f'_{I,J}$ and $g'_{J,I}$ the same way $f$ and $g$ are from $f_{I,J}$ and $g_{J,I}$. Suppose $I, I' \in B(M)$. Then we have
\begin{align}
\begin{split}
\sum_{J \in B(N)} g_{J,I'}(\delta) f_{I,J} &= \sum_{J \in B(N)} {g'_{J,I'}}(\delta) f'_{I,J}.
\end{split}
\end{align}
When $I$ and $I'$ are of different types, the left side is zero because $f$ and $g$ are $\delta$-interleaving morphisms, and all the summands on the right side are zero by definition of $f'$ and $g'$. When $I$ and $I'$ are of the same type, the equality follows from Lemma \ref{typesplit}. This means that $g'(\delta) f' = g(\delta) f$. We also have $f'(\delta) g' = f(\delta) g$, so $f'$ and $g'$ are $\delta$-interleaving morphisms. In particular,~$f'$ and $g'$ are $\delta$-interleaving morphisms when restricted to the components of $M$ and $N$ of a fixed type. If we can show that $f'$ and $g'$ induce a $(2n-1)\delta$-matching on each of the mentioned components, we will have proved Theorem \ref{main}. In other words, we have reduced the problem to the case where all the intervals in $B(M)$ and $B(N)$ are of the same type.

For a decorated number $a^*$, let $u(a^*) = a$ if $a \neq \pm \infty$ and $u(a^*) = 0$ otherwise. Let $a = (a_1, a_2, \dots, a_n)$ be a decorated point. We define $P(a)$ to be the number of the decorated numbers $a_i$ decorated with $+$, and we also define $\alpha(a) = \sum_{1 \leq i \leq n} u(a_i)$. What we really want to look at are rectangles and not decorated points by themselves, so we define $P(R) = P(\textrm{min}_R) + P(\textrm{max}_R)$ and $\alpha(R) = \alpha(\textrm{min}_R) + \alpha(\textrm{max}_R)$ for any rectangle $R$.
Define an order $\leq_\alpha$ on decorated points given by $a \leq_\alpha b$ if either
\begin{itemize}
\item $\alpha(a) < \alpha(b)$, or
\item $\alpha(a) = \alpha(b)$ and $P(a) \leq P(b)$
\end{itemize}
This defines a preorder. In other words, it is transitive ($R \leq_\alpha S \leq_\alpha T$ implies $R \leq_\alpha T$) and reflexive ($R \leq_\alpha R$ for all $R$). We write $R <_\alpha S$ if $R \leq_\alpha S$ and not $R \geq_\alpha S$.

The order $\leq_\alpha$ is one of the most important ingredients in the proof. The point is that if there is a nonzero morphism from $\mathbb{I}^R$ to $\mathbb{I}^S(\epsilon)$ and $R \leq_\alpha S$, then $R$ and $S$ have to be close to each other. If $\epsilon = 0$,~$R$ and $S$ actually have to be equal. This `closeness property' is expressed in Lemma \ref{close}, and is also exploited in Lemma \ref{nonzero}. Finally, in the proof of Lemma \ref{lemmaHall}, we make sure that we only have to deal with morphisms $g_{J,I'}(\delta) \circ f_{I,J}$ for $I \leq_\alpha I'$ and not $I >_\alpha I'$, so that our lemmas can be applied.

In Figure \ref{fig:shift} we see two rectangles $R = (0,4) \times (0,4)$ and $S = (2,5) \times (2,5)$. There is no nonzero morphism from $\mathbb{I}^R$ to $\mathbb{I}^S$ or $\mathbb{I}^{S(1)}$, because $\textrm{min}_R < \textrm{min}_{S(\epsilon)}$ for all $\epsilon < 2$. This is connected to the fact that $\alpha(R) = 8 < 14 = \alpha(S)$, which can be interpreted to mean that $R$ is `further down and to the left' than $S$. The point of including $P(\alpha)$ in the definition of $\alpha$ is that e.g.~$(a,b]$ is a tiny bit `further to the right' than $[a,b)$, and this is a subtlety that $P$ recognizes, and that matters in the proofs of Lemmas \ref{close} and \ref{nonzero}.

\begin{figure}
\centering
\includegraphics[scale=0.5]{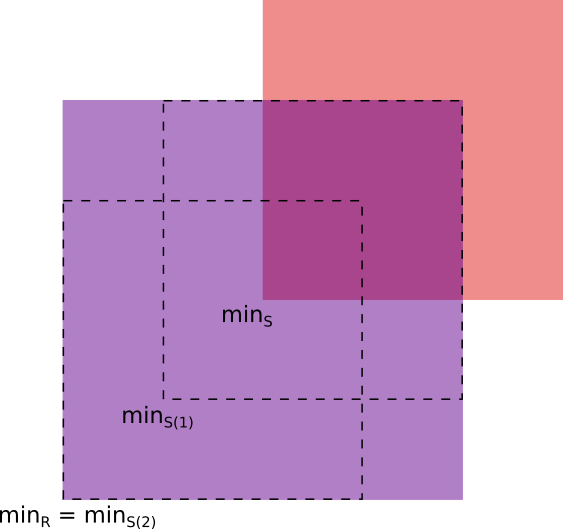}
\caption{Rectangles $R = (0,4) \times (0,4)$ (purple),~$S = (2,5) \times (2,5)$ (pink),~$S(1) = (1,4) \times (1,4)$ (dotted border), and $S(2) = (0,3) \times (0,3)$ (dotted border).\label{fig:shift}}
\end{figure}

\begin{lemma}
\label{close}
Let $R$,~$S$, and $T$ be rectangles of the same type with $R \leq_\alpha T$. Suppose there are nonzero morphisms $\chi: \mathbb{I}^R \rightarrow \mathbb{I}^S(\epsilon)$ and $\psi: \mathbb{I}^S \rightarrow \mathbb{I}^T(\epsilon)$. Then $\mathbb{I}^S$ is $(2n-1)\epsilon$-interleaved with either $\mathbb{I}^R$ or $\mathbb{I}^T$.
\end{lemma}

\begin{proof}
Since $\chi \neq 0$, we have
\begin{itemize}
\item $\textrm{min}_S \leq \textrm{min}_R + \epsilon$
\item $\textrm{max}_S \leq \textrm{max}_R + \epsilon$.
\end{itemize}
This follows from Corollary \ref{minmaxineq}.

Suppose $\mathbb{I}^R$ and $\mathbb{I}^S$ are not $(2n-1)\epsilon$-interleaved. Then either $\textrm{min}_S + (2n-1)\epsilon \ngeq \textrm{min}_R$ or $\textrm{max}_S + (2n-1)\epsilon \ngeq \textrm{max}_R$; let us assume the latter. (The former is similar.) In this case, there is an $m$ such that $\textrm{max}_{S_m} < \textrm{max}_{R_m} - (2n-1)\epsilon$. For $i \neq m$, we have $\textrm{max}_{S_i} \leq \textrm{max}_{R_i} + \epsilon$ by the second bullet point. We get
\begin{align}
\begin{split}
\sum_{1 \leq i \leq n} u(\textrm{max}_{S_i}) &\leq \left(\sum_{1 \leq i \leq n} u(\textrm{max}_{R_i})\right) - (2n-1)\epsilon + (n-1)\epsilon \\
&= \left(\sum_{1 \leq i \leq n} u(\textrm{max}_{R_i})\right) - n\epsilon.
\end{split}
\end{align}

The first bullet point gives us
\begin{equation}
\sum_{1 \leq i \leq n} u(\textrm{min}_{S_i}) \leq \left(\sum_{1 \leq i \leq n} u(\textrm{min}_{R_i})\right) + n\epsilon,
\end{equation}
so we get $\alpha(S) \leq \alpha(R)$. If the inequality is strict, we have $S <_\alpha R$. If not, we have
\begin{itemize}
\item $u(\textrm{min}_{S_i}) = u(\textrm{min}_{R_i}) + \epsilon$ for all $i$
\item $u(\textrm{max}_{S_i}) = u(\textrm{max}_{R_i}) + \epsilon$ for $i \neq m$
\item $u(\textrm{max}_{S_m}) = u(\textrm{max}_{R_m}) - (2n-1)\epsilon$.
\end{itemize}
Because of the inequalities $\textrm{min}_S \leq \textrm{min}_R + \epsilon$ and $\textrm{max}_S \leq \textrm{max}_R + \epsilon$ (recall that these are inequalities of decorated points with the poset structure we defined earlier), we have $P(\textrm{min}_{S_i}) \leq P(\textrm{min}_{R_i})$ for all $i$ and $P(\textrm{max}_{S_i}) \leq P(\textrm{max}_{R_i})$ for $i \neq m$. But since $\textrm{max}_{S_m} < \textrm{max}_{R_m} - (2n-1)\epsilon$, we have $P(\textrm{max}_{S_m}) < P(\textrm{max}_{R_m})$, so $S <_\alpha R$. Similarly, we can prove $T <_\alpha S$ if $\mathbb{I}^S$ and $\mathbb{I}^T$ are not $(2n-1)\epsilon$-interleaved, so we have $T <_\alpha R$, which is a contradiction.
\end{proof}

\begin{lemma}
\label{nonzero}
Let $R$,~$S$, and $T$ be rectangles of the same type with $R$ and $T$ $(4n-2)\epsilon$-significant and $\alpha(R) \leq \alpha(T)$. Suppose there are nonzero morphisms $\chi: \mathbb{I}^R \rightarrow \mathbb{I}^S(\epsilon)$ and $\psi: \mathbb{I}^S \rightarrow \mathbb{I}^T(\epsilon)$. Then $\psi(\epsilon) \circ \chi \neq 0$.
\end{lemma}

The constant $(4n-2)$ can be improved on for $n > 1$, but since the constant $(2n-1)$ in Lemma \ref{close} is optimal, strengthening Lemma \ref{nonzero} will not help us get a better constant in Theorem \ref{main}.

\begin{proof}
Suppose that $\chi$ and $\psi$ are nonzero, but $\psi(\epsilon) \circ \chi = 0$. We have

\begin{itemize}
\item $\textrm{min}_R + 2\epsilon \geq \textrm{min}_{T}$
\item $\textrm{min}_{R_m} + 2\epsilon \geq \textrm{max}_{T_m}$ for some $m$
\item $\textrm{max}_R + 2\epsilon \geq \textrm{max}_{T}$
\item $\textrm{max}_{R_m} \geq \textrm{max}_{T_m} + (4n-4)\epsilon$.
\end{itemize}
The first and third statements hold because $\chi, \psi \neq 0$. (See Corollary \ref{minmaxineq}.) The second is equivalent to $\textrm{min}_R \nless \textrm{max}_{T(2\epsilon)}$. If this did not hold, $R$ and $T(2\epsilon)$ would intersect, and  $\psi(\epsilon) \circ \chi$ would be nonzero in this intersection, which is a contradiction. The fourth statement follows from the second and the fact that $R$ is $(4n-2)\epsilon$-significant.

Since $T$ is $(4n-2)\epsilon$-significant,~$\textrm{min}_{T} + (4n-2)\epsilon < \textrm{max}_{T}$. Thus the second bullet point implies that $\textrm{min}_{R_m} + 2\epsilon > \textrm{min}_{T_m} + (4n-2)\epsilon$. The first point gives $\textrm{min}_{R_i} \geq \textrm{min}_{T_i} - 2\epsilon$ for $i \neq m$. In a similar fashion, we get from the last two points that $\textrm{max}_{R_m} \geq \textrm{max}_{T} + (4n-4)\epsilon$ and $\textrm{max}_{R_i} \geq \textrm{max}_{T_i} - 2\epsilon$ for $i \neq m$. From all this, we get
\begin{align}
\begin{split}
\alpha(R) &= \sum_{1 \leq i \leq n} u(\textrm{min}_{R_i}) + u(\textrm{max}_{R_i}) \\
&\geq u(\textrm{min}_{T_m}) + u(\textrm{max}_{T_m}) + 2(4n-4)\epsilon + \sum_{i \neq m} (u(\textrm{min}_{T_i}) + u(\textrm{max}_{T_i}) - 4\epsilon) \\
&= \alpha(T) + (4n-4)\epsilon \\
&\geq \alpha(T).
\end{split}
\end{align}
Equality only holds if $u(\textrm{min}_{T_m}) + (4n-2)\epsilon = u(\textrm{max}_{T_m})$,~$u(\textrm{min}_{R_m}) + (4n-2)\epsilon = u(\textrm{max}_{R_m})$, and $n = 1$. This means that $R = R_1 = T = T_1 = [u(\textrm{min}_R), u(\textrm{min}_R) + 2\epsilon]$. As we see,~$R \cap T(2\epsilon) = [u(\textrm{min}_R), u(\textrm{min}_R)] \neq \varnothing$, so $\psi(\epsilon) \circ \chi \neq 0$.
\end{proof}

We define a function $\mu$ by
\begin{equation}
\mu(I) = \{J \in B(N) \mid I \textrm{ and } J \textrm{ are } (2n-1)\delta \textrm{-interleaved}\}
\end{equation}
for $I$ in $B(M)$. In other words,~$\mu(I)$ contains all the intervals that can be matched with $I$ in a $(2n-1)\delta$-matching. Let $I \in B(M)$ be $(4n-2)\delta$-significant, and pick $p \in \mathbb{R}^n$ such that $p, p + (4n-2)\delta \in I$. Then,~$p + (2n-1)\delta \in J$ for every $J \in \mu(I)$. Since $M$ and $N$ are p.f.d., this means that $\mu(I)$ is a finite set. For $A \subset B(M)$, we write $\mu(A) = \bigcup_{I \in A} \mu(I)$.

\begin{lemma}
\label{lemmaHall}
Let $A$ be a finite subset of $B(M)$ containing no $(4n-2)\delta$-trivial elements. Then $|A| \leq |\mu(A)|$.
\end{lemma}

Before we prove Lemma \ref{lemmaHall}, we show that it implies that there is a $(2n-1)\delta$-matching between $B(M)$ and $B(N)$ and thus completes the proof of Theorem \ref{main}.

Let $G_\mu$ be the undirected bipartite graph on $B(M) \sqcup B(N)$ with an edge between $I$ and $J$ if $J \in \mu(I)$. Observe that $G_\mu$ is the same as the graph $G_{(2n-1)\delta}$ we defined when we gave the graph theoretical definition of an $\epsilon$-matching (in this case, $(2n-1)\delta$-matching) in section \ref{defs}. Following that definition, a $(2n-1)\delta$-matching is a matching in $G_\mu$ that covers the set of all $(4n-2)\delta$-significant elements in $B(M)$ and $B(N)$.

For a subset $S$ of a graph $G$, let $A_G(S)$ be the neighbourhood of $S$ in $G$, that is, the set of vertices in $G$ that are adjacent to at least one vertex in $S$. We now apply Hall's marriage theorem \cite{hall35} to bridge the gap between Lemma \ref{lemmaHall} and the statement we want to prove about matchings.
\begin{theorem}[Hall's theorem]
\label{hall}
Let $G$ be a bipartite graph on bipartite sets $X$ and $Y$ such that $A_G(\{x\})$ is finite for all $x \in X$. Then the following are equivalent:
\begin{itemize}
\item for all $X' \subset X$, $|X'| \leq |A_G(X')|$
\item there exists a matching in $G$ covering $X$.
\end{itemize}
\end{theorem}

One of the two implications is easy, since if $|X'| > |A_G(X')|$ for some $X' \subset X$, then there is no matching in $G$ covering $X'$. It is the other implication we will use, namely that the first statement is sufficient for a matching in $G$ covering $X$ to exist.

Letting $X$ be the set of $(4n-2)\delta$-significant intervals in $B(M)$ and $Y$ be $B(N)$, Hall's theorem and Lemma \ref{lemmaHall} give us a matching $\sigma$ in the graph $G_\mu$ covering all the $(4n-2)\delta$-significant elements in $B(M)$.\footnote{Strictly speaking, Lemma \ref{lemmaHall} says nothing about infinite $A$, but the case with $A$ countably infinite follows from the finite cases. Each interval in $A$ contains a rational point, so since $M$ is p.f.d., the cardinality of $A$ is at most finite times countably infinite, which is countable. Thus we have covered all the possible cases.} By symmetry, we also have a matching $\tau$ in $G_\mu$ covering all the $(4n-2)\delta$-significant elements in $B(N)$. Neither of these is necessarily a $(2n-1)\delta$-matching, however, as each of them only guarantees that all the $(4n-2)$-significant intervals in one of the barcodes are matched. We will use $\sigma$ and $\tau$ to construct a $(2n-1)\delta$-matching. This construction is similar to one used to prove the Cantor-Bernstein theorem \cite[pp. 110-111]{theBook}.

Let $H$ be the undirected bipartite graph on $B(M) \sqcup B(N)$ for which the set of edges is the union of the edges in the matchings $\sigma$ and $\tau$. Let $C$ be a connected component of $H$. Suppose the submatching of $\sigma$ in $C$ does not cover all the $(4n-2)\delta$-significant elements of $C$. Then there is a $(4n-2)\delta$-significant $J \in C \cap B(N)$ that is not matched by $\sigma$. If we view $\sigma$ and $\tau$ as partial bijections $\sigma: B(M) \nrightarrow B(N)$ and $\tau: B(N) \nrightarrow B(M)$, we can write the connected component of $J$, which is $C$, as $\{J, \tau(J), \sigma(\tau(J)), \tau(\sigma(\tau(J))), \dots\}$. Either this sequence is infinite, or it is finite, in which case the last element is $(4n-2)\delta$-trivial. In either case, we get that the submatching of $\tau$ in $C$ covers all $(4n-2)\delta$-significant elements in $C$.

By this argument, there is a $(2n-1)\delta$-matching in each connected component of $H$. We can piece these together to get a $(2n-1)\delta$-matching in $B(M) \sqcup B(N)$, so Lemma \ref{lemmaHall} completes the proof of Theorem \ref{main}.

\begin{proof}[Proof of Lemma \ref{lemmaHall}]
Because $\leq_\alpha$ is a preorder, we can order $A = \{I_1, I_2, \dots, I_r\}$ so that $I_i \leq_\alpha I_{i'}$ for all $i \leq i'$. Write $\mu(A) = \{J_1, J_2, \dots, J_s\}$. For $I \in B(M)$, we have
\begin{align}
\begin{split}
\phi_{\mathbb{I}^{I},2\delta} &= \pi_I(2\delta) g(\delta) f|_I \\
&= \pi_I(2\delta) \left(\sum_{J \in B(N)} g|_J \pi_J\right)(\delta) f|_I \\
&= \sum_{J \in B(N)} \pi_I(2\delta) g|_J(\delta) \pi_J(\delta) f|_I \\
&= \sum_{J \in B(N)} g_{J,I}(\delta) f_{I,J}.
\end{split}
\end{align}
Also,~$\sum_{J \in B(N)} g_{J,I'}(\delta) f_{I,J} = 0$ for $I \neq I' \in B(M)$, since $\phi_{M,2\delta}$ is zero between different components of $M$. Lemma \ref{close} says that if $g_{J,I'}(\delta) f_{I,J} \neq 0$ and $I \leq_\alpha I'$, then $J$ is $(2n-1)\delta$-interleaved with either $I$ or $J'$. This means that if $i < i'$, then
\begin{align}
\begin{split}
0 &= \sum_{J \in B(N)} g_{J,I_{i'}}(\delta) f_{I_i,J}\\
&= \sum_{J \in \mu(A)} g_{J,I_{i'}}(\delta) f_{I_i,J},
\end{split}
\end{align}
as $g_{J,I_{i'}}(\delta) f_{I_i,J} = 0$ for all $J$ that are not $(2n-1)\delta$-interleaved with either $I_i$ or $I_{i'}$. Similarly,
\begin{align}
\begin{split}
\phi_{\mathbb{I}^{I_i},2\delta} &= \sum_{J \in B(N)} g_{J,I_i}(\delta) f_{I_i,J}\\
&= \sum_{J \in \mu(A)} g_{J,I_i}(\delta) f_{I_i,J}.
\end{split}
\end{align}
Writing this in matrix form, we get
\[
\left[
\begin{smallmatrix}
g_{J_1,I_1}(\delta) &  \dots  & g_{J_s,I_1}(\delta) \\
\vdots & \ddots & \vdots \\
g_{J_1,I_r}(\delta) &  \dots  & g_{J_s,I_r}(\delta)
\end{smallmatrix}
\right]
\left[
\begin{smallmatrix}
f_{I_1,J_1} &  \dots  & f_{I_r,J_1} \\
\vdots & \ddots & \vdots \\
f_{I_1,J_s} &  \dots  & f_{I_r,J_s}
\end{smallmatrix}
\right]
=
\left[
\begin{smallmatrix}
\phi_{M_{\mathbb{I}^{I_1},2\delta}} & ? & \dots  & ? \\
0 & \phi_{M_{\mathbb{I}^{I_2},2\delta}} & \dots  & ? \\
\vdots & \vdots & \ddots & \vdots \\
0 & 0 & \dots  & \phi_{M_{\mathbb{I}^{I_r},2\delta}}
\end{smallmatrix}
\right].
\]
That is, on the right-hand side we have the internal morphisms of the $I_i$ on the diagonal, and $0$ below the diagonal.

Recall that a morphism between rectangle modules can be identified with a $k$-endomorphism, and that in our notation,~$f_{I,J}$ and $g_{J,I}$ are given by multiplication by $w(I,J)$ and $w(J,I)$, respectively. For an arbitrary morphism $\psi$ between rectangle modules, we introduce the notation $w(\psi) = c$ if $\psi$ is given by multiplication by $c$, and $0$ otherwise. A consequence of Lemma \ref{nonzero} is that $w(g_{J,I_{i'}}(\delta) f_{I_i,J}) = w(g_{J,I_{i'}})w(f_{I_i,J}) = w(J,I_i)w(I_{i'},J)$ whenever $I_i \leq_\alpha I_{i'}$, in particular if $i \leq i'$. We get
\begin{align}
\begin{split}
1 &= w\left(\phi_{\mathbb{I}^{I},2\delta}\right)\\
&= w\left(\sum_{J \in \mu(A)} g_{J,I_i}(\delta) f_{I_i,J}\right)\\
&= \sum_{J \in \mu(A)} w(g_{J,I_i}(\delta) f_{I_i,J})\\
&= \sum_{J \in \mu(A)} w(J,I_i) w(I_i,J),\\
\end{split}
\end{align}
and similarly $0 = \sum_{J \in \mu(A)} w(J,I_{i'}) w(I_i,J)$ for $i \leq i'$. Again we can interpret this as a matrix equation:
\[
\begin{bmatrix}
w(J_1,I_1) &  \dots  & w(J_s,I_1) \\
\vdots & \ddots & \vdots \\
w(J_1,I_r) & \dots  & w(J_s,I_r) \\
\end{bmatrix}
\begin{bmatrix}
w(I_1,J_1) & \dots  & w(I_r,J_1) \\
\vdots & \ddots & \vdots \\
w(I_1,J_s) &  \dots  & w(I_r,J_s) \\
\end{bmatrix}
=
\begin{bmatrix}
1 & ? & \dots  & ? \\
0 & 1 & \dots  & ? \\
\vdots & \vdots & \ddots & \vdots \\
0 & 0 & \dots  & 1
\end{bmatrix}.
\]
That is, the right-hand side is an $r \times r$ upper triangular matrix with $1$'s on the diagonal. The right-hand side has rank $|A|$ and the left-hand side has rank at most $|\mu(A)|$, so the lemma follows immediately from this equation.
\end{proof}

\subsection{Block decomposable modules}

Next, we prove stability for block decomposable modules, which, as explained in \cref{zigReeb}, implies stability for zigzag modules and Reeb graphs. Let $\mathbb{R}_+^2 = \{(x,y) \in \mathbb{R}^2 \mid x + y \geq 0\}$.
\begin{definition}
A \textbf{triangle} is a nonempty set $T$ of the form $\{(x,y) \in \mathbb{R}^2 \mid x < a, y < b\} \cap \mathbb{R}_+^2$ for some $(a,b) \in (\mathbb{R} \cup \{\infty\})^2$ with $a + b > 0$.
\end{definition}
It follows that triangles are intervals. For a triangle $T = \{(x,y) \in \mathbb{R}^2 \mid x < a, y < b\} \cap \mathbb{R}_+^2$, we write $\textrm{max}_T = (a,b) \in (\mathbb{R} \cup \{\infty\})^2$. If $T$ is bounded,~$\textrm{max}_T$ is the maximal element in the closure of $T$, as illustrated in Figure \ref{fig:triangle}. A \textit{triangle decomposable module} is an interval decomposable $\R^2$-module whose barcode only contains triangles. Observe that triangles correspond exactly to blocks of the form $(a,b)_{BL}$ under the poset isomorphism between $\R^{\text{op}}\times \R$ and $\R^2$ flipping the $x$-axis.

\begin{figure}
\centering
\includegraphics[scale=0.6]{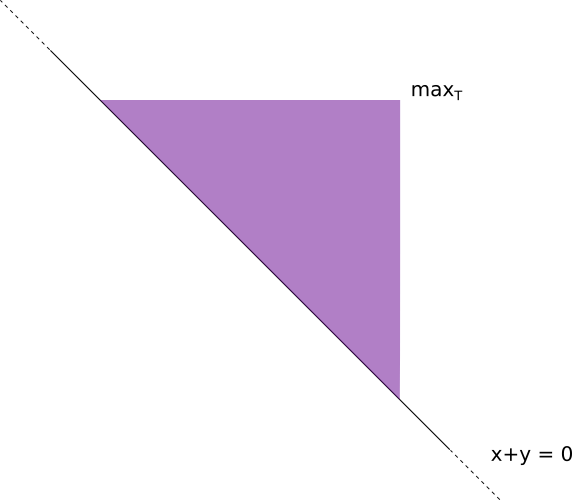}
\caption{A bounded triangle $T$.\label{fig:triangle}}
\end{figure}

\begin{theorem}
\label{triangle}
Let $M$ and $N$ be $\delta$-interleaved triangle decomposable modules. Then there is a $\delta$-matching between $B(M)$ and $B(N)$.
\end{theorem}

To prove this, we can split the triangles into sets of different `types', as we did with the rectangles. We get four different types of triangles $T$, depending on whether $\textrm{max}_T$ is of the form $(a,b)$,~$(\infty,b)$,~$(a,\infty)$, or $(\infty,\infty)$ for $a,b \in \mathbb{R}$. Now a result analogous to Lemma \ref{typesplit} holds, implying that it is enough to show Theorem \ref{triangle} under the assumption that the barcodes only contain intervals of a single type. The case in which the triangles are bounded is the hardest one, and the only one we will prove. So from now on, we assume all triangles to be bounded.

Again, we reuse parts of the proof of Theorem \ref{main}. For $I \in B(M)$, we define $\nu(I) = \{J \in B(N) \mid I \textrm{ and } J \textrm{ are } \delta \textrm{-interleaved}\}$. The discussion about Hall's theorem is still valid, so we only need to prove the analogue of Lemma \ref{lemmaHall} for $\nu$. Define $\alpha(T) = m_1+m_2$, where $\textrm{max}_T = (m_1,m_2)$. The only things we need to complete the proof of the analogue of Lemma \ref{lemmaHall} for triangle decomposable modules are the following analogues of Lemmas \ref{close} and \ref{nonzero}:

\begin{lemma}
\label{close triangle}
Let $R$,~$S$, and $T$ be triangles with $\alpha(R) \leq \alpha(T)$. Suppose there are morphisms $f: \mathbb{I}^R \rightarrow \mathbb{I}^S(\epsilon)$ and $g: \mathbb{I}^S \rightarrow \mathbb{I}^T(\epsilon)$ such that $g(\epsilon) \circ f \neq 0$. Then $\mathbb{I}^S$ is $\epsilon$-interleaved with either $\mathbb{I}^R$ or $\mathbb{I}^T$.
\end{lemma}

\begin{lemma}
\label{nonzero triangle}
Let $R$,~$S$, and $T$ be triangles with $T$ $2\epsilon$-significant and $\alpha(R) \leq \alpha(T)$. Suppose there are nonzero morphisms $f: \mathbb{I}^R \rightarrow \mathbb{I}^S(\epsilon)$ and $g: \mathbb{I}^S \rightarrow \mathbb{I}^T(\epsilon)$. Then $g(\epsilon) \circ f \neq 0$.
\end{lemma}

\begin{proof}[Proof of Lemma \ref{close triangle}]
Suppose $\mathbb{I}^R$ and $\mathbb{I}^S$ are not $\epsilon$-interleaved. Then $\textrm{max}_S \ngeq \textrm{max}_R - \epsilon$. But at the same time,~$\textrm{max}_R \geq \textrm{max}_S - \epsilon$, which gives $\alpha(R) > \alpha(S)$. Assuming that $\mathbb{I}^S$ and $\mathbb{I}^T$ are not $\epsilon$-interleaved, either, we also get $\alpha(S) > \alpha(T)$. Thus $\alpha(R) > \alpha(T)$, a contradiction.
\end{proof}

\begin{proof}[Proof of Lemma \ref{nonzero triangle}]
For all triangles $I$, we treat $\textrm{min}_I$ and $\textrm{max}_I$ as undecorated points. We have $\textrm{max}_T - \epsilon \leq \textrm{max}_S$ and $\textrm{max}_S - \epsilon \leq \textrm{max}_R$, so $\textrm{max}_T - 2\epsilon \leq \textrm{max}_R$. Because $T$ is $2\epsilon$-significant,~$\textrm{max}_T - 2\epsilon - \epsilon' \in \mathbb{R}_+^2$ for some $\epsilon' > 0$. Combining these facts, we get $\textrm{max}_T - 2\epsilon - \epsilon' \in R$, so $(g(\epsilon) \circ f)_{\textrm{max}_T - 2\epsilon - \epsilon'} \neq 0$.
\end{proof}

\cref{triangle} implies $d_B(M,N) = d_I(M,N)$ for block decomposable $M$ and $N$ such that $B(M)$ and $B(N)$ only have blocks of the form $(a,b)_{BL}$ (so no closed or half-closed blocks). Our proof technique extends easily to prove the same equality for all block decomposable $M$ and $N$. In fact, $d_B(M,N) \leq d_I(M,N)$ in the case where all the intervals in the barcodes are of the form $[a,b]_{BL}$ follows from \cref{free} below with $n=2$ by the correspondence $[a,b]_{BL} \leftrightarrow \langle(-a,b)\rangle$, while the two cases with half-open blocks are both essentially the algebraic stability theorem. In the end we could stitch the cases together by something similar to \cref{typesplit} and the discussion following it. We omit the details, and anyway the closed and half-open cases are taken care of in \cite{zigzag}. Thus, either by appealing to previous work for the other cases or using our own methods, we get
\begin{theorem}
\label{thmBlock}
Let $M$ and $N$ be block decomposable modules. If $M$ and $N$ are $\delta$-interleaved, there exists a $\delta$-matching between $B(M)$ and $B(N)$.
\end{theorem}

\subsection{Free modules}

\begin{definition}
We define a \textbf{free interval} as an interval of the form $\langle p \rangle := \{q \mid q \geq p\} \subset \mathbb{R}^n$.
\end{definition}
For a free interval $R$, we define $\textrm{min}_R$ by $R = \langle \textrm{min}_R \rangle$.\footnote{This makes $\textrm{min}_R$ an undecorated point, while we have previously defined $\textrm{min}_-$ as decorated points, but this does not matter, as we will not need decorated points in this subsection.} We define a \textit{free $\mathbb{R}^n$-module} as an interval decomposable module whose barcode only contains free intervals. It is easy to see that free intervals are rectangles, so it follows from Theorem \ref{main} that $d_B(M,N) \leq (2n-1)d_I(M,N)$ for free modules $M$,~$N$. But because of the geometry of free modules, this result can be strengthened.

\begin{theorem}
\label{free}
Let $M$ and $N$ be free $\delta$-interleaved $\mathbb{R}^n$-modules with $n \geq 2$. Then there is a $(n-1)\delta$-matching between $B(M)$ and $B(N)$.
\end{theorem}

We already did most of the work while proving Theorem \ref{main}, and there are some obvious simplifications. Firstly, free intervals are $\epsilon$-significant for all $\epsilon \geq 0$. Secondly, for all nonzero $f: \mathbb{I}^R \rightarrow \mathbb{I}^S$ and $g: \mathbb{I}^S \rightarrow \mathbb{I}^T$ with $R$,~$S$,~$T$ free,~$gf$ is nonzero. For $I \in B(M)$, define $\nu(I) = \{J \in B(N) \mid I \textrm{ and } J \textrm{ are } (n-1)\delta \textrm{-interleaved}\}$. By the arguments in the proof of Theorem \ref{main}, we only need to prove Lemma \ref{lemmaHall} with $\mu$ replaced by $\nu$. Lemmas \ref{close} and \ref{nonzero} still hold for free modules, but we need to sharpen Lemma \ref{close}.

\begin{lemma}
Let $R$,~$S$, and $T$ be free intervals with $R \leq_\alpha T$. Suppose there are morphisms $0 \neq f: \mathbb{I}^R \rightarrow \mathbb{I}^S(\epsilon)$ and $0 \neq g: \mathbb{I}^S \rightarrow \mathbb{I}^T(\epsilon)$. Then $\mathbb{I}^S$ is $(n-1)\epsilon$-interleaved with either $\mathbb{I}^R$ or $\mathbb{I}^T$.
\end{lemma}

\begin{proof}
In this proof, we treat $\textrm{min}_I$ and $\textrm{max}_I$ as undecorated points for all free intervals $I$, so that we can add them. We have $\textrm{min}_S \leq \textrm{min}_R + \epsilon$. Suppose $\mathbb{I}^R$ and $\mathbb{I}^S$ are not $(n-1)\epsilon$-interleaved. Then $\textrm{min}_S + (n-1)\epsilon \ngeq \textrm{min}_R$, so for some $m$, we must have $\textrm{min}_{S_m} < \textrm{min}_{R_m} - (n-1)\epsilon$. We get
\begin{align}
\begin{split}
\alpha(S) &= \sum_{1 \leq i \leq n} \textrm{min}_{S_i} \\
&< \textrm{min}_{R_m} - (n-1)\epsilon + \sum_{i \neq m} \left( \textrm{min}_{R_i} + \epsilon \right) \\
&= \sum_{1 \leq i \leq n} \textrm{min}_{R_i} \\
&= \alpha(R).
\end{split}
\end{align}
We can also prove that $\alpha(T) < \alpha(S)$ if $\mathbb{I}^S$ and $\mathbb{I}^T$ are not $(n-1)\epsilon$-interleaved, so we have $\alpha(T) < \alpha(R)$, a contradiction.
\end{proof}

\section{Counterexamples to a general algebraic stability theorem}
\label{tight}

Theorem \ref{main} gives an upper bound of $(2n-1)$ on $d_B/d_I$ for rectangle decomposable modules that increases with the dimension. An obvious question is whether it is possible to improve this constant, or if for each $C < 2(n-1)$ there exist pairs $M,N$ of modules for which $d_B(M,N) > Cd_I(M,N)$, in which case the bound is optimal. We know that $d_B(M,N) \geq d_I(M,N)$ for any $M$ and $N$ whenever the bottleneck distance is defined, so for $n = 1$, the constant is optimal. For $n > 1$, however, it turns out that the equality $d_B(M,N) = d_I(M,N)$ does not always hold, and the geometry becomes more confusing when $n$ increases. In dimension $2$, we give an example of rectangle decomposable modules $M$ and $N$ with $d_B(M,N) = 3d_I(M,N)$ in Example \ref{example3}, which means that the bound is optimal for $n=2$, as well. This is a counterexample to a conjecture made in a previous version of \cite{zigzag} which claims that interval decomposable $\mathbb{R}^n$-modules $M$ and $N$ such that $B(M)$ and $B(N)$ only contain convex intervals are $\epsilon$-matched if they are $\epsilon$-interleaved.

\begin{example}
\label{example2}
Let $B(M) = \{I_1, I_2, I_3 \}$\footnote{Here we use subscripts to index different intervals, not to indicate projections, as we did earlier.} and $B(N) = \{J \}$, where
\begin{itemize}
\item $I_1 = (-3,1) \times (-1,3)$
\item $I_2 = (-1,3) \times (-3,1)$
\item $I_3 = (-1,1) \times (-1,1)$
\item $J = (-2,2) \times (-2,2)$.
\end{itemize}

\begin{figure}
\centering
\includegraphics[scale=0.5]{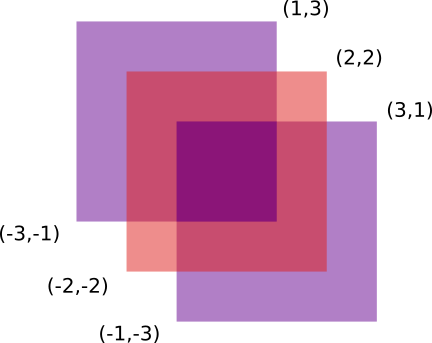}
\caption{$M$ and $N$.~$I_1$ and $I_2$ are the light purple squares,~$I_3$ is deep purple, and $J$ is pink.\label{fig:C=2}}
\end{figure}

See Figure \ref{fig:C=2}. We can define $1$-interleaving morphisms $f: M \rightarrow N(1)$ and $g: N \rightarrow M(1)$ by letting $w(I_1,J) = w(I_2,J) = w(I_3,J) = w(J,I_1) = w(J,I_2) = 1$ and $w(J,I_3) = -1$, where $w$ is defined as in the proof of Theorem \ref{main}. On the other hand, in any matching between $B(M)$ and $B(N)$ we have to leave either $I_1$ or $I_2$ unmatched, and they are $\epsilon$-significant for all $\epsilon < 4$. In fact, any possible matching between $B(M)$ and $B(N)$ is a $2$-matching. Thus $d_I(M,N) = 1$ and $d_B(M,N) = 2$.
\end{example}

A crucial point is that even though $w(I_1,J)$,~$w(J,I_2)$,~$w(I_2,J)$, and $w(J,I_1)$ are all nonzero, both $g_{J,I_2} \circ f_{I_1, J}$ and $g_{J,I_1} \circ f_{I_2, J}$ are zero. To do the same with one-dimensional intervals, we would have to shrink $I_1$ and $I_2$ so much that they no longer would be $2$-significant (see Lemma \ref{nonzero}), and then they would not need to be matched in a $1$-matching. This shows how the geometry of higher dimensions can allow us to construct examples that would not work in lower dimensions.

Next, we give an example of rectangle decomposable $\mathbb{R}^2$-modules $M$ and $N$ such that $d_B(M,N) = 3d_I(M,N)$, proving that our upper bound of $2(n-1)$ is the best possible for $n=2$.

\begin{example}
\label{example3}
Let $B(M) = \{I_1, I_2, I_3 \}$ and $B(N) = \{J_1, J_2, J_3 \}$, where
\begin{itemize}
\item $I_1 = (0,10) \times (1,11)$
\item $I_2 = (0,12) \times (-1,11)$
\item $I_3 = (2,10) \times (1,9)$
\item $J_1 = (1,11) \times (0,10)$
\item $J_2 = (1,9) \times (0,12)$
\item $J_3 = (-1,11) \times (2,10)$.
\end{itemize}
The rectangles in $B(M)$ and $B(N)$ are shown in Figure \ref{fig:dB3}.

\begin{figure}
\centering
\includegraphics[scale=0.18]{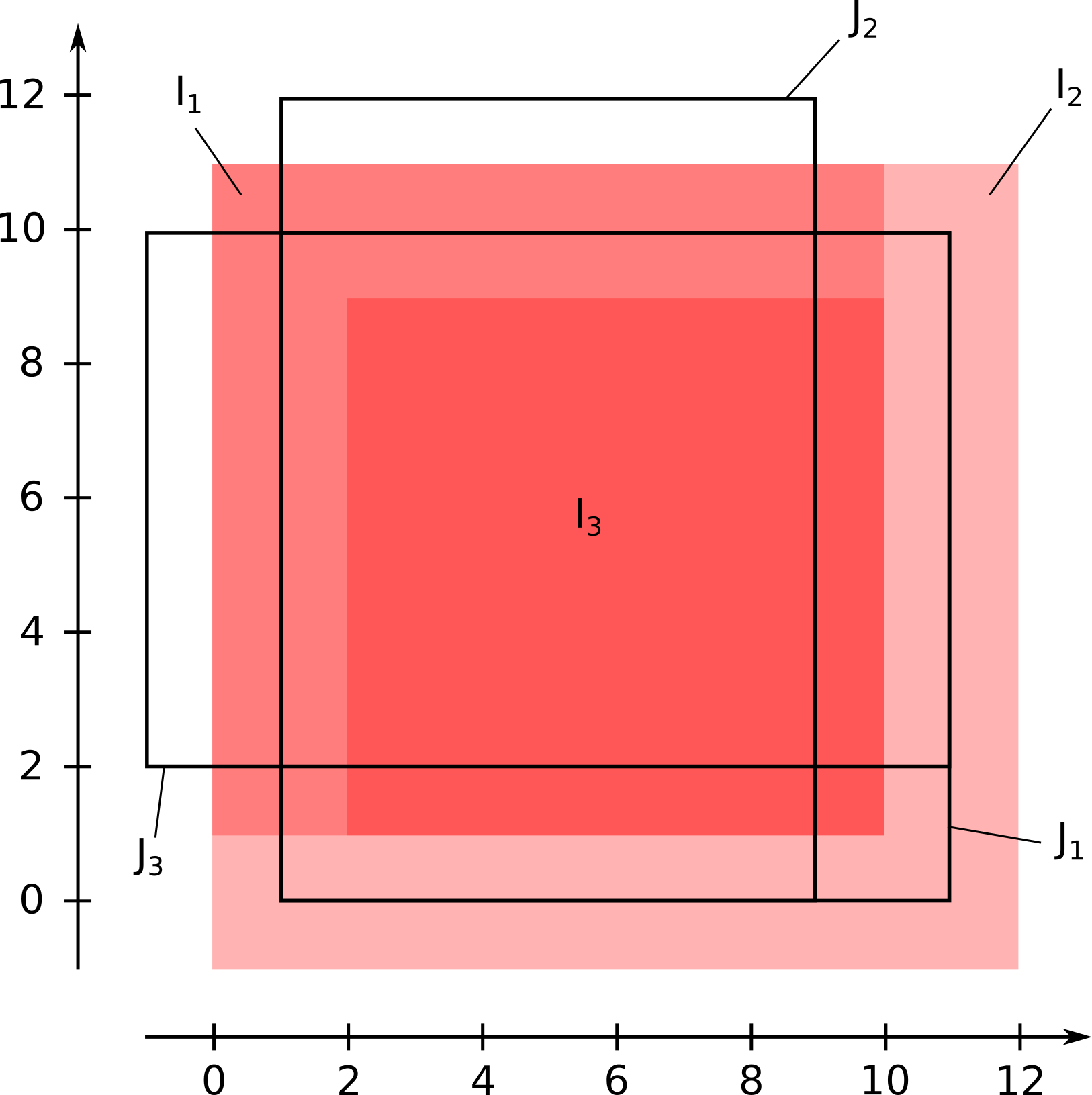}
\caption{$I_1$, $I_2$, and $I_3$ are the filled pink rectangles, and $J_1$, $J_2$, and $J_3$ are the black rectangles without fill.\label{fig:dB3}}
\end{figure}

We give an example of $1$-interleaving morphisms $f$ and $g$ that we write on matrix form. In the first matrix,~$w(I_i,J_j)$ is in row $i$, column $j$. In the second,~$w(J_j, I_i)$ is in row $j$, column $i$.

\begin{align}
&f: 
\begin{bmatrix}
1&1&1 \\
1&1&0 \\
1&0&1
\end{bmatrix}, &g:
\begin{bmatrix}
-1&1&1 \\
1&0&-1 \\
1&-1&0
\end{bmatrix}.
\end{align}

This means that $M$ and $N$ are $1$-interleaved, but they are not $\epsilon$-interleaved for any $\epsilon < 1$, so $d_I(M,N) = 1$.

Let $\epsilon < 3$. We see that the difference between $\textrm{max}_{I_2} = (12,11)$ and $\textrm{max}_{J_2} = (9,12)$ is $3$ in the first coordinate, so $I_2$ and $J_2$ are not $\epsilon$-interleaved, and thus they cannot be matched in an $\epsilon$-matching. In fact, $I_i$ and $J_j$ cannot be matched in an $\epsilon$-matching for any $i,j \in \{2,3\}$ by similar arguments. Since $I_2$ and $I_3$ cannot both be matched with $J_1$, one of them has to be left unmatched, but since both $I_2$ and $I_3$ are $6$-significant, this means that there is no $\epsilon$-matching between $B(M)$ and $B(N)$. On the other hand, any bijection between $B(M)$ and $B(N)$ is a $3$-matching, so $d_B(M,N) = 3$.
\end{example}

There is a strong connection between $n$-dimensional rectangle decomposable modules and $2n$-dimensional free modules. This is related to the fact that we need $2n$ coordinates to determine an $n$-dimensional rectangle, and also $2n$ coordinates to determine a $2n$-dimensional free interval. The following example illustrates this connection, as we simply rearrange the coordinates of $\textrm{min}_{R}$, $\textrm{max}_{R}$ for all rectangles $R$ involved in Example \ref{example3} to get $4$-dimensional free modules with similar properties as in Example \ref{example3}.

\begin{example}

\label{example3free}
Let $B(M) = \{I_1, I_2, I_3 \}$ and $B(N) = \{J_1, J_2, J_3 \}$, where
\begin{itemize}
\item $I_1 = \langle (0,1,10,11) \rangle$
\item $I_2 = \langle (0,-1,12,11) \rangle$
\item $I_3 = \langle (2,1,10,9) \rangle$
\item $J_1 = \langle (1,0,11,10) \rangle$
\item $J_2 = \langle (1,0,9,12) \rangle$
\item $J_3 = \langle (-1,2,11,10) \rangle$.
\end{itemize}
(Compare with the intervals $I_i$ and $J_j$ in Example \ref{example3}.) We have $1$-interleaving morphisms defined the same way as in Example \ref{example3}. Just as in that example, we can deduce that there is nothing better than a $3$-matching between $B(M)$ and $B(N)$, so $d_B(M,N) = 3$ and $d_I(M,N) = 1$.
\end{example}
As a consequence of this example, we get that our upper bound of $d_B/d_I \leq n-1$ for free $n$-dimensional modules cannot be improved on for $n=4$.

\section{Relation to the complexity of calculating interleaving distance}
\label{complexity}

The interleaving distance between arbitrary persistence modules is on the surface not easy to find, as naively trying to construct interleaving morphisms can quickly lead you to consider a complicated set of equations for which it is not clear that one can decide if there is a solution in polynomial time. For $\R$-modules, however, the interval decomposition theorem plus the algebraic stability theorem gives us a polynomial time algorithm to compute $d_I$: decompose the modules into intervals and find the bottleneck distance. Since $d_I = d_B$, this gives us the interleaving distance. When it exists, one can compute the bottleneck distance in polynomial time also in two dimensions \cite{dey2018computing}, but the approach fails for general $\R^n$-modules already at the first step, as we do not have a nice decomposition theorem. But in the recent proof that calculating interleaving distance is NP-hard \cite{bjerkevik2018computing}, it is the failure of the second step that is exploited. Specifically, a set of modules that decompose nicely into interval modules (\emph{staircase modules}, to be precise) is constructed, but for these, $d_I$ and $d_B$ are different. It turns out that calculating $d_I$ for these corresponds to deciding whether \emph{CI problems} are solvable, which is shown to be NP-hard.

Though rectangle modules are not considered in the NP-hardness proof, they have similar properties to staircase modules,\footnote{The only significant difference in this setting is that in a fixed dimension, rectangle modules are defined by a limited number of coordinates, or ``degrees of freedom'', while there is no such restriction on staircase modules even in dimension $2$.} and Example \ref{example3} is essentially a CI problem with a corresponding pair of modules. Importantly, it shows that $d_I = d_B$ does not hold in general for modules corresponding to CI problems. This crucial observation, which appeared first in a preprint of this paper \cite{myself}, opened the door to proving NP-hardness of calculating $d_I$ by the approach used in \cite{bjerkevik2018computing}.

In \cite{bjerkevik2018computing}, it is also shown that \emph{$c$-approximating} $d_I$ is NP-hard for $c<3$, where an algorithm is said to $c$-approximate $d_I$ if it returns a number in the interval $[d_I(M,N), cd_I(M,N)]$ for any input pair $M$, $N$ of modules. Whether the approach by CI problems can be used to prove hardness of $c$-approximation for $c\geq 3$ is closely related to \cref{main}. It can be shown that if $d_B(M,N) \leq cd_I(M,N)$ for any pair $M$, $N$ of rectangle decomposable modules, the same holds for staircase modules, and therefore there is a polynomial time algorithm $c$-approximating $d_I$ for these, meaning that the strategy of going through CI problems will not give a proof that $c$-approximation of $d_I$ is NP-hard. On the other hand, if one can find an example of rectangle decomposable modules $M$ and $N$ such that $d_B(M,N) = cd_I(M,N)$ for $c>3$, one might be able to use that to increase the constant $3$ in the approximation hardness result. Thus there is a strong link between stability of rectangle decomposable modules and the only successful method so far known to the author of determining the complexity of computing or approximating multiparameter interleaving distance.

\section{Acknowledgements}
I would like to thank my supervisors Gereon Quick and Nils Baas for invaluable support and help. I would also like to thank Peter Landweber for detailed comments on several drafts of this text, Steve Oudot for feedback on the first arXiv version and Magnus Bakke Botnan for interesting discussions.

\bibliography{bibfile}
\bibliographystyle{plain}
\end{document}